\documentclass[a4paper,11pt,reqno]{amsart}
\usepackage{amsmath}
\usepackage[T1]{fontenc}
\usepackage{amssymb}
\usepackage{amsthm}
\usepackage{color}
\usepackage[lmargin=2.5 cm,rmargin=2.5 cm,tmargin=3.5cm,bmargin=2.5cm,
paper=a4paper]{geometry}

\newcommand{\Ab}{\mathbf A}
\newcommand{\Fb}{\mathbf F}
\newcommand{\Bb}{\mathbf B}
\newcommand{\ab}{\mathbf a}
\newcommand{\R}{\mathbb R}
\newcommand{\C}{\mathbb C}

\DeclareMathOperator{\IM}{Im}

\DeclareMathOperator{\curl}{curl}
\DeclareMathOperator{\Div}{div}

\newtheorem{thm}{Theorem}[section]
\newtheorem{prop}[thm]{Proposition}

\newtheorem{corol}[thm]{Corollary}

\theoremstyle{remark}
\newtheorem{rem}[thm]{Remark}

\numberwithin{equation}{section}

\title[3D Ginzburg-Landau  functional]{The ground state energy of the three dimensional
Ginzburg-Landau model in the mixed phase}

\address{$^{\mathrm A}$ Lebanese University, Department of Mathematics, Hadath, Beirut, Lebanon}
\address{$^{\mathrm B}$ University of Aarhus, Department of
Mathematics, 1530 NyMunkegade, 8000 Aarhus C}
\email{ayman.kashmar@liu.edu.lb}

\author[A. Kachmar]{Ayman Kachmar\,$^{\mathrm a,\mathrm b}$}

\date{\today}

\begin{document}
\begin{abstract}
We consider the Ginzburg-Landau functional defined over a bounded and smooth three dimensional  domain.
Supposing that the strength of the applied magnetic field
 varies between the first and second  critical fields, in such a way that $H_{C_1}\ll H\ll H_{C_2}$,
we estimate the ground state energy to leading order as the Ginzburg-Landau parameter tends to infinity.
\end{abstract}
\maketitle

\section{Introduction and main result}
We consider a bounded and open set $\Omega\subset\R^3$ with smooth boundary. We
 suppose that $\Omega$ models a superconducting sample subject to an applied external magnetic field.
The energy of
the sample is
given by the Ginzburg-Landau functional,
\begin{multline}\label{eq-3D-GLf}
\mathcal E^{\rm 3D}(\psi,\Ab)=\mathcal
E_{\kappa,H}^{\rm 3D}(\psi,\Ab)=
\int_\Omega\bigg(
|(\nabla-i\kappa
H\Ab)\psi|^2+\frac{\kappa^2}2(1-|\psi|^2)^2\bigg)\,dx\\
+\kappa^2H^2\int_{\R^3}|\curl\Ab-\beta|^2\,dx\,.
\end{multline}
Here $\kappa$ and $H$ are two positive parameters, the wave function
(order parameter)
$\psi\in H^1(\Omega;\C)$, the induced magnetic potential $\Ab\in
\dot{H}^1_{\Div,\Fb}(\R^3)$, where $\dot H^1_{\Div,\Fb}(\R^3)$ is the space introduced in
\eqref{eq-3D-hs} below. Finally, $\beta$ is the external magnetic
field that we choose constant, $\beta=(0,0,1)$.

Let $\dot H^1(\R^3)$ be the homogeneous sobolev space, i.e. the closure
of $C_c^\infty(\R^3)$ under the norm $u\mapsto\|u\|_{\dot
  H^1(\R^3)}:=\|\nabla u\|_{L^2(\R^3)}$. Let further
  $\Fb(x)=(-x_2/2,x_1/2,0)$. Clearly $\Div \Fb=0$.

We define the space,
\begin{equation}\label{eq-3D-hs}
\dot H^1_{\Div,\Fb}(\R^3)=\{\Ab~:~\Div \Ab=0\,,\quad~{\rm and}\quad
\Ab-\Fb\in \dot H^1(\R^3)\}\,.
\end{equation}

Critical points $(\psi,\Ab)\in H^2(\Omega;\C)\times \dot
H^1_{\Div,\Fb}(\R^3)$ of $\mathcal E ^{\rm 3D}$
satisfy the Ginzburg-Landau equations,
\begin{equation}\label{eq-3D-GLeq}
\left\{
\begin{array}{lll}
-(\nabla-i\kappa H\Ab)^2\psi=\kappa^2(1-|\psi|^2)\psi&{\rm in}& \Omega
\\
\curl^2\Ab=-\displaystyle\frac1{\kappa H}\IM(\overline{\psi}\,(\nabla-i\kappa
H\Ab)\psi)\mathbf 1_\Omega&{\rm in}& \R^3\\
\nu\cdot(\nabla-i\kappa H\Ab)\psi=0&{\rm on}&\partial\Omega\,,
\end{array}\right.\end{equation}
where $\mathbf 1_\Omega$ is the characteristic function of the domain $\Omega$, and $\nu$ is the pointing interior unit normal vector  of $\partial\Omega$.

For a solution $(\psi,\Ab)$ of \eqref{eq-3D-GLeq}, the function $\psi$
describes the superconducting properties of the material and
$H\curl\Ab$ gives the induced magnetic field. The number $\kappa$ is
a material parameter, and the number $H$ is the intensity of a constant
magnetic field externally applied to the sample.

In the mathematics literature, Type~II superconductors usually
correspond to the limit $\kappa\to\infty$, see \cite{FH-b, SS}. In
this regime one distinguishes three critical values $H_{C_1}$,
$H_{C_2}$ and $H_{C_3}$ for the applied field. Those critical fields
  are roughly described as follows. If $H<H_{C_1}$, the material is in
  the superconducting phase. Mathematically, this corresponds to
  $|\psi|>0$ for any minimizer $(\psi,\Ab)$ of \eqref{eq-3D-GLf}.
If $H_{C_1}<H<H_{C_2}$, the magnetic field penetrates the sample in
quantized vortices (corresponding to zeros of $\psi$). If $H_{C_2}<H<H_{C_3}$,
superconductivity is confined to the surface of the sample
(corresponding to $|\psi|$ very small in the bulk). Finally, if
$H>H_{C_3}$, superconductivity is lost, which is reflected by $\psi=0$
everywhere in $\Omega$. In this paper, we will focus on the regime
when the applied magnetic field varies between $H_{C_1}$ and $H_{C_2}$. In the scaling we choose in this paper, this regime corresponds to $\ln\kappa/\kappa\ll H\ll\kappa$ as $\kappa\to\infty$. Here, if $a(\kappa)$ and $b(\kappa)$ are two positive functions, the notation $a(\kappa)\ll b(\kappa)$ means that $a(\kappa)/b(\kappa)\to 0$ as $\kappa\to\infty$.

In the case of  two dimensional domains, which correspond to
infinite cylindrical superconducting samples, there exists
a quite satisfactory analysis of the critical fields $H_{C_1}$,
$H_{C_2}$ and $H_{C_3}$. As we can not give an exhaustive list of
references,
we invite the reader to see the monographs \cite{FH-b, SS}, where a
detailed review of the material is present.  Still in the two dimensional setting, the most accurate available characterization of the critical field $H_{C_2}$ is given in \cite{FK, FK1}.

The situation is less understood in three dimensions, especially the regime of  magnetic fields close to the first critical field $H_{C_1}$. For a superconductor occupying a ball domain, a candidate for the expression of the critical field $H_{C_1}$ is given in \cite{ABM}.  Related results are obtained for superconducting shells in \cite{CS}. For general domains,
the analysis of the critical field $H_{C_3}$ started   in \cite{LP}, then a sharp characterization of $H_{C_3}$ is given in \cite{FH3d}. In the papers \cite{Al, Pa}, it is proved that superconductivity is  confined to the surface of the domain, provided that
 magnetic field is close to and below $H_{C_3}$. A fine characterization of the critical field $H_{C_2}$ together with leading order estimates of the ground state energy in large magnetic fields are recently obtained in \cite{FK3D, FKP}.
This paper is  complementary to those in \cite{FK3D,FKP}.

The ground state energy of the functional in \eqref{eq-3D-GLf} is defined as follows,
\begin{equation}\label{eq-3D-gs}
C_0(\kappa,H)=\inf\big\{
\mathcal E^{\rm 3D}(\psi,\Ab)~:~(\psi,\Ab)\in H^1(\Omega;\C)\times \dot H^1_{\Div,\Fb}(\R^3)\big\}\,.
\end{equation}
The main result of this paper is Theorem~\ref{thm-3D-main} below. It is a generalization of an analogous  result proved for the two-dimensional functional in \cite{SS00}.

\begin{thm}\label{thm-3D-main}
Suppose that  the magnetic field $H$ is a function of $\kappa$ and satisfies
$$\frac{\ln\kappa}{\kappa}\ll H\ll\kappa\,,\quad{\rm as}\quad \kappa\to\infty\,.$$
 Then, the ground state energy in \eqref{eq-3D-gs} satisfies,
\begin{equation}\label{eq-3D-thm}
C_0(\kappa,H)=|\Omega|\,\kappa H\ln\sqrt{\frac{\kappa}{H}}+o\bigg(\kappa H\ln\sqrt{\frac{\kappa}{H}} \,\bigg)\,,\quad{\rm as}\quad \kappa\to\infty\,.
\end{equation}
\end{thm}

As immediate consequences of Theorem~\ref{thm-3D-main} we obtain that, if $(\psi,\Ab)$ is a minimizer of
\eqref{eq-3D-GLf}, then the induced magnetic field $\curl \Ab$ is close to the applied magnetic field $\beta$, and that the magnitude of the order parameter $|\psi|$ is close to $1$ almost everywhere in $\Omega$. The physical meaning of this is that the applied magnetic field penetrates the sample almost everywhere and concentrates along `vortex lines'. On these vortex lines the order parameter $\psi$ is expected to have  zeros (this is not rigorously proved in this paper), but away of them, the sample remains in the superconducting phase $(|\psi|$ is close to $1$). Therefore,  the regime considered in Theorem~\ref{thm-3D-main} corresponds to what is actually named in the physics literature as the {\it mixed phase}.

In the course of the proof of Theorem~\ref{thm-3D-main}, we obtain the following conclusions as immediate corollaries.

\begin{corol}\label{corol-main}
Under the assumptions made in Theorem~\ref{thm-3D-main},
if $(\psi,\Ab)\in H^1(\Omega;\C)\times \dot H^1_{\Div}(\R^3)$ is a minimizer of the energy in \eqref{eq-3D-GLf}, then, as $\kappa\to\infty$,
\begin{align}
&\curl\Ab-\beta\to 0\quad{\rm in}~H^1(\R^3;\R^3)\,,\\
&e_{\kappa,H}(\psi,\Ab)\to dx\quad{\rm in}~\mathcal D'(\Omega)\,,\\
&\mu_{\kappa,H}(\psi,\Ab)\to (0,0,dx)\quad {\rm in}~\mathcal D'(\Omega;\R^3)\,.
\end{align}
Here, $dx$ is the Lebesgue measure in $\Omega$,  the measure $e_{\kappa,H}(\psi,\Ab)$ and the current  $\mu_{\kappa, H}(\psi,\Ab)$ are respectively,
\begin{align}
&e_{\kappa,H}(\psi,\Ab)=\frac{\left(|(\nabla-i\kappa H\Ab)\psi|^2+\frac{\kappa^2}2(1-|\psi|^2)^2\right)}{\kappa H\ln\sqrt{\kappa/H}}\,dx\,,\\
&\mu_{\kappa,H}(\psi,\Ab)=\curl\bigg(-\displaystyle\frac1{\kappa H}\IM(\overline{\psi}\,(\nabla-i\kappa
H\Ab)\psi)\bigg)+\curl\Ab\,.
\end{align}
\end{corol}

In two dimensions, $\mu_{\kappa,H}$ is a measure and it is proved that it gives the density of vortices, hence it is called
the {\it vorticity measure}, see \cite{SS}.

The proof of Theorem~\ref{thm-3D-main} is obtained as follows. First we start by the analysis of an approximate problem in a
`large' cube. The cube geometry allows us to link this problem to another two dimensional problem in a square. The later is analyzed using tools from \cite{SS}.

Using a ground state of the approximate problem, we   construct a test configuration whose energy provides an upper bound of the ground state energy $C_0(\kappa,H)$.
As a consequence of this upper bound, we obtain that, for a minimizer $(\psi,\Ab)$ of \eqref{eq-3D-GLf}, the induced magnetic field $\curl\Ab$ is close to the applied field $\beta$ in $L^2$-norm. Using this and the regularity of the $\curl$-$\Div$ system in $\R^3$, we get an estimate of $\Ab-\Fb$ in $C^{0,1/2}$-norm.

The  {\it a priori} estimates obtained for minimizers allow us to determine a lower bound of the energy that matches with the obtained upper bound. Actually, we use the `semi-classical' localization techniques developed in \cite{FH-b}
to reduce the problem to that of the approximate problem in a cube. Then the analysis of the later problem is used to obtain the matching lower bound.

An interesting aspect of the analysis is that we do not use constructions involving vortices, i.e. we do not localize the set where
$\{x\in\Omega~:~|\psi(x)|\leq 1/2\}$ (as this is  certainly difficult in three dimensions).  This is a significant difference between the strategy of our proof and the one given in \cite{SS00} for the two-dimensional functional. However, the construction of `vortex-balls' for the two dimensional functional `implicitly'  appears in the analysis of the three dimensional approximate  problem, as we refer to results of \cite{SS, SS00}. In the context of the Ginzburg-Landau model, the implementation of `semi-classical' techniques to address situations where vortices exist seems rather new.

The analysis presented in Section~\ref{sec:LP} combined with a recently proved estimate in \cite{FK3D} enables us to prove
a theorem of independent interest (Theorem~\ref{thm-f(b)} below), which concerns the asymptotic behavior of a limiting constant appearing in \cite{SS02},  thereby answering a question raised by the authors of the aforementioned paper.

The paper is organized as follows. Section~\ref{sec:LP} is devoted to the analysis of the approximate problem. In Section~\ref{sec:ub}, an upper bound of the ground state energy is obtained. In Section~\ref{sec:ap}, interesting
estimates are obtained for minimizers of \eqref{eq-3D-GLf}. Section~\ref{sec:lb} is devoted to the proof of the lower
bound.

\paragraph{\it Remark on notation:}
\begin{itemize}
\item The letter $C$ denotes a positive constant that is independent of the parameters $\kappa$ and $H$, and whose value  may change from line to line.
\item If $a(\kappa)$ and $b(\kappa)$ are two functions with $b(\kappa)\not=0$, we write $a(\kappa)\sim b(\kappa)$ if $a(\kappa)/b(\kappa)\to1$ as $\kappa\to\infty$.
\end{itemize}

\section{The approximate problem}\label{sec:LP}
\subsection{Two dimensional energy}
Let $K=(-1/2,1/2)\times(-1/2,1/2)$ be a square of unit side length,  $h_{\rm ex}$ and $\varepsilon$ be two positive parameters. Consider the functional defined for all $u\in H^1(K;\C)$,
\begin{equation}\label{eq-2D-f}
E^{\rm 2D}(u)=\int_K\left(|(\nabla-ih_{\rm ex}\Ab_0)u|^2+\frac1{2\varepsilon^2}(1-|u|^2)^2\right)\,dx\,.
\end{equation}
Here $\Ab_0$ is the vector potential,
\begin{equation}\label{eq-Ab0}
\Ab_0(x_1,x_2)=\frac12(-x_2,x_1)\,,\quad (x_1,x_2)\in\R^2\,,\end{equation}
whose $\curl$ is equal to $1$.

Notice that the functional $E^{\rm 2D}$ is a simplified version of the full Ginzburg-Landau functional considered in \cite{SS00}, as the
magnetic potential in \eqref{eq-2D-f} is given and {\it not} an unknown of the problem.

We introduce the ground state energy,
\begin{equation}\label{eq-2D-gse}
m_0(h_{\rm ex},\varepsilon)=\inf\{E^{\rm 2D}(u)~:~u\in H^1(K;\C)\}\,.\end{equation}

Since $E^{\rm 2D}$ is bounded from below, there exists a ground state (minimizer) associated to $m_0(h_{\rm ex},\varepsilon)$. If $u$ is such a ground state, then it results from a standard application of the maximum principle that,
\begin{equation}\label{eq-mprinciple}
|u|\leq1\quad{\rm in }\quad K\,.
\end{equation}

Consider the regime of magnetic fields $h_{\rm ex}$ as in Theorem~\ref{thm-SS00} below.
We can obtain a lower bound of $m_0(h_{\rm ex},\varepsilon)$ (or rather of $E^{\rm 2D}(u)$, with $u$ a ground state) exactly as in \cite[Section~8.2]{SS}, by using a scaling argument that reduces the situation to magnetic fields of lower order (precisely of order $|\ln\varepsilon|$). In this way, we get the following theorem.

\begin{thm}\label{thm-SS00}
Assume that $h_{\rm ex}$ is a function of $\varepsilon$ such that
$$|\ln\varepsilon|\ll h_{\rm ex}\ll \frac1{\varepsilon^2}\,,\quad{\rm as}~\varepsilon\to0\,.$$
Then the ground state energy $m_0(h_{\rm ex},\varepsilon)$ satisfies,
$$m_0(h_{\rm ex},\varepsilon)\geq h_{\rm ex}\ln\frac1{\varepsilon\sqrt{h_{\rm ex}}}\big(1+o(1)\big)\,,
$$
as $\varepsilon\to 0$.
\end{thm}

Minimization of the functional $E^{\rm 2D}$ over `magnetic periodic' functions appears naturally as well. Let us introduce the
following space,
\begin{multline}\label{space-p}
E_{h_{\rm ex}}=\{u\in H^1_{\rm loc}(\R^2;\C)~:~u(x_1+1,x_2)=e^{ih_{\rm ex}x_2/2}u(x_1,x_2)\,,\\u(x_1,x_2+1)=e^{-ih_{\rm ex}x_1/2}u(x_1,x_2)\}\,,
\end{multline}
together with the ground state energy,
\begin{equation}\label{eq-mp}
m_{\rm p}(h_{\rm ex},\varepsilon)=\inf\{E^{\rm 2D}(u)~:~u\in E_{h_{\rm ex}}\}\,.
\end{equation}

\begin{thm}\label{thm-SS-p}
Assume that $h_{\rm ex}$ is a function of $\varepsilon$ such that
$$|\ln\varepsilon|\ll h_{\rm ex}\ll \frac1{\varepsilon^2}\,,\quad{\rm as}~\varepsilon\to0\,.$$
Then the ground state energy $m_{\rm p}(h_{\rm ex},\varepsilon)$ satisfies,
$$m_{\rm p}(h_{\rm ex},\varepsilon)=h_{\rm ex}\ln\frac1{\varepsilon\sqrt{h_{\rm ex}}}\big(1+o(1)\big)\,,
$$
as $\varepsilon\to 0$.
\end{thm}
\begin{proof}
Since the restriction of a function in $E_{h_{\rm ex}}$ to $K$ is a function in $H^1(K)$, we get that $m_{\rm p}(h_{\rm ex},\varepsilon)\geq m_0(h_{\rm ex},\varepsilon)$, where $m_0(h_{\rm ex},\varepsilon)$ is the ground state energy in \eqref{eq-2D-gse}. Theorem~\ref{thm-SS00} then gives us a lower bound of $m_{\rm p}(h_{\rm ex},\varepsilon)$.

We prove the upper bound by computing the energy of a test function $u$
constructed in \cite{AyS}.  Let $N$ be the largest positive integer satisfying $N\leq \sqrt{h_{\rm ex}/2\pi}<N+1$. Divide
the square $K$ into $N^2$ disjoint squares $(K_j)_{0\leq j\leq N^2-1}$ each of side length equal to $1/N$ and center $a_j$.

Let $h$ be the unique solution of the problem,
$$\left\{
\begin{array}{lll}
-\Delta h+h_{\rm ex}=2\pi\delta_{a_0}&{\rm in}&K_0\\
\displaystyle\frac{\partial h}{\partial\nu}=0&{\rm on}&\partial K_0\\
\displaystyle\int_{K_0}h\,dx=0.&&
\end{array}\right.
$$
Here $\nu$ is the unit outward normal vector of $K_0$. By
uniqueness  of $h$ as solution of the aforementioned problem,  $h$
is symmetric with respect to the axes of the square $K_0$ and
hence satisfies periodic conditions on the boundary of $K_0$.
Moreover, the function $v(x)=h(x)-\ln|x-a_0|$ is smooth in $K_0$,
since $-\Delta v+h_{\rm ex}=0$. Consequently, through a scaling
argument, it is easy to check that, as $\varepsilon\to 0$,
\begin{align*}
\int_{K_0\setminus B(a_0,\varepsilon)}|\nabla h|^2\,dx&\leq 2\pi \ln\frac1{\varepsilon N}+\mathcal O(1)\\
&\leq 2\pi\ln\frac1{\varepsilon\sqrt{h_{\rm ex}}}+\mathcal O(1)\,.
\end{align*}
We extend $h$ by periodicity in the square $K$. Let $\phi$ be a function (defined modulo $2\pi$) satisfying
in $K\setminus\{a_{j}~:~0\leq j\leq N^2-1\}$,
$$\nabla\phi=-\nabla^\bot h+h_{\rm ex}\Ab_0\,.$$
Here $\nabla^\bot=(-\partial_{x_2},\partial_{x_1})$ and $\Ab_0$ is the magnetic potential in \eqref{eq-Ab0}.

If $x\in K_0$, let $\rho(x)=\min(1,|x-a_0|/\varepsilon)$. We extend the function $\rho$ by periodicity in the square $K$. We put
$u(x)=\rho(x) e^{i\varphi(x)}$ for all $x\in K$. Then $u$ can be extended as  a function in the space $E_{h_{\rm ex}}$ in \eqref{space-p},
see \cite[Lemma~5.11]{Ay} for details.

The energy of $u$ is easily computed, since $u$ is `magnetic periodic'. Actually,
\begin{align*}
E^{\rm 2D}(u)&=N^2\times \int_{K_0}\left(\rho^2|\nabla h|^2+|\nabla\rho|^2+\frac1{2\varepsilon^2}(1-\rho)^2\right)\,dx\\
&\leq N^2\times\left(2\pi\ln\frac1{\varepsilon\sqrt{h_{\rm ex}}}+\mathcal O(1)\right)\,.
\end{align*}
Since, $N=\sqrt{h_{\rm ex}/2\pi}\big(1+o(1)\big)$ as $\varepsilon\to0$, and $m_{\rm p}(h_{\rm ex},\varepsilon)\leq
E^{\rm 2D}(u)$, we deduce that,
$$m_{\rm p}(h_{\rm ex},\varepsilon)\leq
h_{\rm ex}\ln\frac1{\varepsilon\sqrt{h_{\rm ex}}}\big(1+o(1)\big)\,,$$
as $\varepsilon\to0$.
\end{proof}

Since $m_0(h_{\rm ex},\varepsilon)\leq m_{\rm p}(h_{\rm ex},\varepsilon)$, we get as a corollary of Theorems~\ref{thm-SS00} and \ref{thm-SS-p}:

\begin{corol}\label{corol-SS00}
Let $m_0(h_{\rm ex},\varepsilon)$ be the ground state energy introduced in \eqref{eq-2D-gse} above. Suppose that $h_{\rm ex}$ is a function of $\varepsilon$ and $|\ln\varepsilon|\ll h_{\rm ex}\ll1/\varepsilon^2$ as $\varepsilon\to0$. Then,
$$m_{0}(h_{\rm ex},\varepsilon)=
h_{\rm ex}\ln\frac1{\varepsilon\sqrt{h_{\rm ex}}}\big(1+o(1)\big)\,,$$
as $\varepsilon\to0$.
\end{corol}

Theorem~\ref{thm-SS-p} serves in answering a question of independent interest arising in \cite{SS02}. Consider two constants
$b\in(0,1)$ and $R>0$. Let $K_R=(-R/2,R/2)\times(-R/2,R/2)$. If $u\in H^1(K_R)$, we define the energy,
$$ F_{K_R}(u)=\int_{K_R}\left(b|\nabla-i\Ab_0)u|^2+\frac12(1-|u|^2)^2\right)\,dx\,,$$
together with the ground state energy,
$$e_{\rm p}(b,R)=\inf\{F_{K_R}(u)~:~u\in E_R\}\,.$$
Here $\Ab_0$ is the magnetic potential introduced in \eqref{eq-Ab0} and $E_R$ is the space introduced in \eqref{space-p}, (with $h_{\rm ex}=R$). It is proved that, for all $b\in(0,1)$, there exists a constant $f(b)$ such that,
\begin{equation}\label{eq-f(b)}
f(b)=\frac12\lim_{R\to\infty}\frac{e_{\rm p}(b,R)}{R^2}\,.\end{equation}
The limiting constant $f(b)$ appeared in \cite{SS02, AS}, then it is recently studied with different tools in \cite{FK3D}. This limiting constant describes the ground state energy of both two and three dimensional superconductors subject to high
magnetic fields (see \cite{FK3D}).

The behavior of the function $f(b)$ as $b\to1_-$ is analyzed in details in \cite{FK3D}. However, the behavior as $b\to0_+$ remains open. Only a non-optimal estimate on $f(b)$ is given as $b\to0_+$ in \cite{SS02}. Here, using Theorem~\ref{thm-SS-p}
and an estimate in \cite{FK3D}, we describe the leading order asymptotic behavior of $f(b)$ as $b\to0_+$.

\begin{thm}\label{thm-f(b)}
Let $f(b)$ be as defined in \eqref{eq-f(b)}. Then, as $b\to0_+$, $f(b)$ satisfies,
$$f(b)=\frac{b}2\ln\frac1{\sqrt{b}}\,\big(1+o(1)\big)\,.$$
\end{thm}
\begin{rem}
In \cite{SS02}, it is proved that
$$b-\frac{b^2}2\leq f(b)\leq \frac{b}2\ln\frac1{\sqrt{b}}\,\big(1+o(1)\big)\,,$$
 as $b\to0_+$.
\end{rem}

\begin{proof}[Proof of Theorem~\ref{thm-f(b)}]
It is proved in \cite[Theorem~2.1 \& Proposition~2.8]{FK3D} that there exist universal constants $C$ and $R_0$ such that,
\begin{equation}\label{eq-p-f(b)}
\forall ~b\in(0,1)\,,\quad\forall~R\geq R_0\,,\quad \left|2f(b)-\frac{e_{\rm p}(b,R)}{R^2}\right|\leq \frac{C}R\,.\end{equation}
Let $h_{\rm ex}=R^2$ and $\varepsilon=\sqrt{b}/R$. A scaling argument shows  that,
\begin{equation}\label{eq-pp-f(b)}
e_{\rm p}(b,R)=b\,m_{\rm p}(h_{\rm ex},\varepsilon)\,.\end{equation}
We select $R=1/b$ so that as $b\to0_+$ we have $\varepsilon\to0$ and $|\ln\varepsilon|\ll h_{\rm ex}\ll\varepsilon^{-2}$. Theorem~\ref{thm-SS-p} then tells us that,
$$m_{\rm p}(h_{\rm ex},\varepsilon)=h_{\rm ex}\ln\frac1{\varepsilon\sqrt{h_{\rm ex}}}\,\big(1+o(1)\big)\,.$$
We insert this estimate into \eqref{eq-pp-f(b)} then we substitute the values  $h_{\rm ex}=R^2$ and $\varepsilon\sqrt{h_{\rm ex}}=\sqrt{b}$. Finally, inserting the resulting estimate into \eqref{eq-p-f(b)} finishes  the proof of the proposition.
\end{proof}


\subsection{Three dimensional energy}
If $\mathcal D$ is an open set of $\R^3$ and $u\in H^1(\mathcal D;\C)$, we define the energy
\begin{equation}\label{eq-3D-en}
G_{\mathcal D}(u)=\int_{Q_R}\left(b|(\nabla-i\Fb)u|^2+\frac12(1-|u|^2)^2\right)\,dx\,.
\end{equation}
Here $\Fb$ is the magnetic potential,
\begin{equation}\label{eq-Fb}
\Fb(x_1,x_2,x_3)=(-x_2/2,x_1/2,0)\,,\quad (x_1,x_2,x_3)\in\R^3\,,
\end{equation}
whose $\curl$ is equal to $1$.

Let $b$ and $R$ be two positive parameters. Consider a cube $Q_R$ of side length $R$ defined as follows,
\begin{equation}\label{eq-QR}
Q_R=(-R/2,R/2)\times(-R/2,R/2)\times(-R/2,R/2)\,.
\end{equation}
We introduce the ground state energy,
\begin{equation}\label{eq-3D-gse}
M_0(b,R)=\inf\{G_{Q_R}(u)~:~u\in H^1(Q_R;\C)\}\,.
\end{equation}
In the next theorem, we give an asymptotic lower bound of the ground state energy $M_0(b,R)$ as $b\to0$ and $R\to\infty$ simultaneously, in such a way that $\ln(Rb^{-1/2})\ll R^2$.

\begin{thm}\label{thm-3D-gse}
Suppose that the positive parameters $b=b(\epsilon)$ and $R=R(\epsilon)$ are functions of a parameter $\epsilon$ such that,
$$b(\epsilon)\to 0\,,\quad R(\epsilon)\to\infty\,,\quad{\rm and}\quad
\frac1{R(\epsilon)^2}\ln\frac{R(\epsilon)}{\sqrt{b(\epsilon)}}\to0\,,$$
as $\epsilon\to 0$.

Then, the ground state energy $M_0(b,R)$ satisfies,
$$\frac{M_0(b,R)}{R^3}= b\ln\frac1{\sqrt{b}}\big(1+o(1)\big)\,,
$$
as $\epsilon\to0$.
\end{thm}
\begin{proof}
Let $h_{\rm ex}=R^2$ and $\varepsilon=\sqrt{b}/R$. By the assumption on $b$ and $R$, it is easy to see that
$\varepsilon\to 0$ and $|\ln\varepsilon|\ll h_{\rm ex}\ll1/{\varepsilon^2}$.

Consequently, Corollary~\ref{thm-SS00} tells us that the ground state energy $m_0(h_{\rm ex},\varepsilon)$ in \eqref{eq-2D-gse} satisfies,
$$m_0(h_{\rm ex},\varepsilon)= h_{\rm ex}\ln\frac1{\varepsilon\sqrt{h_{\rm ex}}}\big(1+o(1)\big)\,.$$
 We will prove that,
$$M_0(b,R)=bR \,m_0(h_{\rm ex},\varepsilon)\,,$$
which will immediately give us the asymptotic estimate in Theorem~\ref{thm-3D-gse}.

Let $u\in H^1(Q_R;\C)$, $K=(-1/2,1/2)\times(-1/2,1/2)$ and $Q_1=K\times (-1/2,1/2)$. Define the rescaled function
$\widetilde u\in H^1(Q_1;\C)$ as follows,
$$\forall~x\in Q_1\,,\quad \widetilde u(x)=u(Rx)\,.$$
It is easy to check that,
\begin{align*}
G_{Q_R}(u)&=bR\int_{-1/2}^{1/2} \left(\int_K\big(|(\nabla-ih_{\rm ex})\Fb)\widetilde u|^2+\frac1{2\varepsilon^2}(1-|\widetilde u|^2)^2\big)\,dx_\bot\right)dx_3\\
&\geq bR\int_{-1/2}^{1/2}
\left(\int_K\big(|(\nabla_{x_\bot}-ih_{\rm ex})\Fb)\widetilde
u|^2+\frac1{2\varepsilon^2}(1-|\widetilde
u|^2)^2\big)\,dx_\bot\right)dx_3\,.
\end{align*}
Here, if $x=(x_1,x_2,x_3)\in\R^3$, we write $x_\bot=(x_1,x_2)$ and
$\nabla_{x_\bot}=(\partial_{x_1},\partial_{x_2})$. Then, recalling the definition of $m_0(h_{\rm ex},\varepsilon)$, we get,
$$G_{Q_R}(u)\geq bR\int_{-1/2}^{1/2} m_0(h_{\rm ex},\varepsilon)\,dx_3=bR\,m_0(h_{\rm ex},\varepsilon)\,.$$
Taking the infimum over all functions $u\in H^1(Q_R;\C)$, we get that
$M_0(b,R)\geq bR\,m_0(h_{\rm ex},\varepsilon)$.

Let $u_{h_{\rm ex},\varepsilon}$ be a ground state of $E^{\rm 2D}$, i.e. $E^{\rm 2D}(u_{h_{\rm ex},\varepsilon})=m_0(h_{\rm ex},\varepsilon)$. Define the function,
$$u:Q_R\ni x\mapsto u_{h_{\rm ex},\varepsilon}(x_\bot/R)\,.$$
Then, $G_{Q_R}(u)=bR\, E^{\rm 2D}(u_{h_{\rm ex},\varepsilon})$,  thereby showing that
$M_0(b,R)\leq bR\,m_0(h_{\rm ex},\varepsilon)$.
\end{proof}

\section{Upper bound of the energy}\label{sec:ub}

The aim of this section is to give an upper bound on the ground state energy $C_0(\kappa,H)$ in \eqref{eq-3D-gs}.

\begin{thm}\label{thm-ub}
Assume that the magnetic field $H$ satisfies $\ln\kappa/\kappa\ll H\ll \kappa$ as $\kappa\to\infty$. Then the ground state energy
$C_0(\kappa,H)$ in \eqref{eq-3D-gs} satisfies,
\begin{equation}\label{eq-ub}
C_0(\kappa,H)\leq |\Omega|\kappa H\ln\sqrt{\frac{\kappa}{H}}\,\big(1+o(1)\big)\,,\end{equation}
as $\kappa\to\infty$.

Furthermore, there exists a constant $\kappa_0$ such that,  if $\kappa\geq\kappa_0$ and
 $(\psi,\Ab)$ is a minimizer of the functional in $\eqref{eq-3D-GLf}$, then
\begin{equation}\label{eq-Ab=Fb}
\|\curl(\Ab-\Fb)\|_{L^2(\R^3)}\leq \frac{2|\Omega|}{\sqrt{\kappa H}}\sqrt{\ln\sqrt{\frac{\kappa}{H}}}\,.
\end{equation}
\end{thm}
\begin{proof}
Notice that if $(\psi,\Ab)$ is a minimizer of \eqref{eq-3D-GLf}, then $\mathcal E^{\rm 3D}(\psi,\Ab)=C_0(\kappa, H)$.
Consequently, the estimate in \eqref{eq-Ab=Fb} follows immediately from the upper bound in \eqref{eq-ub}.

Let $b=H/\kappa$ and $\ell=\left(\displaystyle\frac{\kappa H}{\ln\kappa}\right)^{1/4}\displaystyle\frac1{\sqrt{\kappa H}}$. Then, as $\kappa\to\infty$, we have,
$$b\ll1\,,\quad \ell\ll1\,,\quad \ell\sqrt{\kappa H}\gg1\,.$$
Let
$h_{\rm ex}=1/\ell^2$ and $\varepsilon=\sqrt{b}\,\ell$. Then, as $\kappa\to\infty$, we have $\varepsilon\ll1$ and
$|\ln\varepsilon|\ll h_{\rm ex}\ll1/\varepsilon^2$.

Recall the ground state energy $m_{\rm p}(h_{\rm ex},\varepsilon)$ and the space $E_{h_{\rm ex}}$
 introduced in  \eqref{eq-mp} and \eqref{space-p} respectively. Let $u\in E_{h_{\rm ex}}$ be a ground state corresponding to
$m_{\rm p}(h_{\rm ex},\varepsilon)$, i.e.
$$\int_{K}\left(|(\nabla-ih_{\rm ex}\Ab_0)u|^2+\frac1{2\varepsilon^2}(1-|u|^2)^2\right)\,dx=m_{\rm p}(h_{\rm ex},\varepsilon)\,.$$
For all $x=(x_\bot,x_3)\in\R^3$, we introduce the function,
$$v(x)=u\big(\ell\sqrt{\kappa H}\,x_\bot\big)\,.$$  Let $(Q_j)$ be a lattice of $\R^3$ generated by the cube,
$$Q=\bigg(-\frac1{2\ell\sqrt{\kappa H}},\frac1{2\ell\sqrt{\kappa H}}\bigg)\times \bigg(-\frac1{2\ell\sqrt{\kappa H}},\frac1{2\ell\sqrt{\kappa H}}\bigg)
\times \bigg(-\frac1{2\ell\sqrt{\kappa H}},\frac1{2\ell\sqrt{\kappa H}}\bigg)\,.$$
It is easy to check that,
$$
\int_{Q}\left(|(\nabla-i\kappa H\Fb)v|^2+\frac{\kappa^2}{2}(1-|v|^2)^2\right)\,dx=\frac1{\ell\sqrt{\kappa H}}\,
m_{\rm p}(h_{\rm ex},\varepsilon)\,.
$$
Here $\Fb$ is the magnetic potential in \eqref{eq-Fb}. Let $\mathcal J=\{Q_j~:~Q_j\cap\partial\Omega\not=\emptyset\}$
and $N={\rm Card}\,\mathcal J$. Then, as $\kappa\to\infty$, we have,
$$N=|\Omega|\times \big(\ell\sqrt{\kappa H}\big)^3\,\big(1+o(1)\big)\,.$$
Recall the functional $\mathcal E^{3D}$ in \eqref{eq-3D-GLf}. We compute the energy of the test configuration $(v,\Fb)$.
Since $\curl \Fb=\beta$ and the function $v$  is magnetic periodic with respect to the lattice $Q_j$, we get,
\begin{align*}
\mathcal E^{3D}(v,\Fb)&=N\times
\int_{Q}\left(|(\nabla-i\kappa H\Fb)v|^2+\frac{\kappa^2}{2}(1-|v|^2)^2\right)\,dx\\
&=N\times \frac1{\ell\sqrt{\kappa H}}m_{\rm p}(h_{\rm ex},\varepsilon)\,.
\end{align*}
We use Theorem~\ref{thm-SS-p}, the definitions of $h_{\rm ex}$ and $\varepsilon$, and the asymptotic behavior of $N$ to get,
$$N\times \frac1{\ell\sqrt{\kappa H}}m_{\rm p}(h_{\rm ex},\varepsilon)
=\kappa H\ln\sqrt{\frac{\kappa}{H}}\,\big(1+o(1)\big)\,,$$
as $\kappa\to\infty$. This proves the upper bound of Theorem~\ref{thm-ub}.
\end{proof}

\section{A priori estimates of minimizers}\label{sec:ap}

The aim of this section is to give {\it a priori} estimates on the solutions of the Ginzburg-Landau equations \eqref{eq-3D-GLeq}. Those estimates play an essential role in controlling the error resulting from various approximations.

The starting point is the following  $L^\infty$-bound  resulting from the maximum principle. Actually,
if $(\psi,\Ab)\in H^1(\Omega;\C)\times \dot H^1_{\Div,\Fb}(\R^3)$ is a solution of \eqref{eq-3D-GLeq}, then
\begin{equation}\label{eq-psi<1}
\|\psi\|_{L^\infty(\Omega)}\leq1\,.
\end{equation}

Next we prove an estimate on the induced magnetic potential.

\begin{prop}\label{prop-ap}
Suppose that the magnetic field $H$ is a function of $\kappa$ such that
$\ln\kappa\ll \kappa H\ll\kappa^2$ as $\kappa\to\infty$. There exist positive constants $\kappa_0$ and $C$
such that, if $\kappa\geq\kappa_0$ and $(\psi,\Ab)\in H^1(\Omega;\C)\times \dot H^1_{\Div,\Fb}(\R^3)$ is a minimizer of the energy in \eqref{eq-3D-GLf}, then,
\begin{align*}
&\|\Ab-\Fb\|_{H^2(\Omega)}\leq \frac{C}{\sqrt{\kappa H}}\sqrt{\ln\sqrt{\frac{\kappa}{H}}}\,,\\
&\|\Ab-\Fb\|_{C^{0,1/2}(\Omega)}\leq\frac{C}{\sqrt{\kappa H}}\sqrt{\ln\sqrt{\frac{\kappa}{H}}} \,.\end{align*}
Here $\Fb$ is the magnetic potential introduced in \eqref{eq-Fb}.
\end{prop}
\begin{proof}
The estimate in $C^{0,1/2}$-norm is a consequence of the Sobolev embedding
of $H^2(\Omega)$ in $C^{0,1/2}(\Omega)$\,.

Notice that it follows from Theorem~\ref{thm-ub} that,
\begin{equation}\label{eq-est4}
\|\curl(\Ab-\Fb)\|_{L^2(\R^3)}\leq \frac{2|\Omega|}{\sqrt{\kappa H}}\sqrt{\ln\sqrt{\frac{\kappa}{H}}}\,,\quad
\|(\nabla-i\kappa H\Ab)\psi\|_{L^2(\Omega)}\leq 2|\Omega|\sqrt{\kappa H}\,\sqrt{\ln\sqrt{\frac{\kappa}{H}}}\,.\end{equation}

Let $a=\Ab-\Fb$. We will prove that $\|a\|_{H^2(\Omega)}\leq \displaystyle\frac{C}{\sqrt{\kappa H}}\sqrt{\ln\sqrt{\frac{\kappa}{H}}}$\,.
Since $\Div a=0$, we get
by regularity of the curl-div system  (see e.g. \cite[Theorem~D.3.1]{FH-b}),
\begin{equation}
\|a\|_{L^6(\R^3)}\leq C\|\curl a\|_{L^2(\R^3)}\,.\label{eq-est3}
\end{equation}
The second equation in \eqref{eq-3D-GLeq} reads as follows,
$$-\Delta a=\frac1{\kappa H}\IM(\overline\psi\,(\nabla-i\kappa H\Ab)\psi)\mathbf 1_\Omega\,.$$
Select a positive constant $M$ such that the open ball  $K=B(0,M)$ contains $\Omega$. By elliptic estimates (see e.g. \cite[Theorem~E.4.2]{FH-b}),
$$\|a\|_{H^2(\Omega)}\leq C(\|a\|_{L^2(K)}+\|\Delta a\|_{L^2(K)})\,.$$
Using the embedding of $L^2(K)$ into $L^6(K)$, the estimate in \eqref{eq-est3} and the bound $|\psi|\leq1$, we get that,
$$\|a\|_{H^2(\Omega)}\leq C\left(\|\curl a\|_{L^2(\R^3)}+\frac{1}{\kappa H}\|(\nabla-i\kappa H)\psi\|_{L^2(\Omega)}\right)\,.$$
Inserting the estimates in \eqref{eq-est4} into this upper bound finishes the proof of the proposition.
\end{proof}

\section{Lower bound of the energy}\label{sec:lb}

In this section, we suppose that $D$ is an open set with smooth boundary such that $D\subset\Omega$. We will give a lower bound of the energy,
\begin{equation}\label{eq-en-E0}
\mathcal E_0(\psi,\Ab;D)=\int_D\left(|(\nabla-i\kappa H\Ab)\psi|^2+\frac{\kappa^2}2(1-|\psi|^2)^2\right)\,dx\,,
\end{equation}
where $(\psi,\Ab)$ is a minimizer of the functional in \eqref{eq-3D-GLf}. The precise statement is the subject of the next theorem.

\begin{thm}\label{thm-lb}
Suppose that the magnetic field $H$ is a function of $\kappa$ such that $\ln\kappa\ll\kappa H\ll\kappa^2$ as $\kappa\to\infty$.
If $(\psi,\Ab)\in H^1(\Omega;\C)\times \dot H^1_{\Div,\Fb}(\R^3)$ is a minimizer of the function in \eqref{eq-3D-GLf}, then,
$$
\mathcal E_0(\psi,\Ab;D)\geq |D|\kappa H\ln\sqrt{\frac{\kappa}{H}}+o\bigg(\kappa H\ln\sqrt{\frac{\kappa}{H}}\,\bigg)\,,$$
as $\kappa\to\infty$. Here $\mathcal E_0(\psi,\Ab;D)$ is introduced in \eqref{eq-en-E0}.
\end{thm}
\begin{proof}
Let $\ell\in(0,1)$ be a parameter (depending on $\kappa$) that will be chosen later in such a way that $(\sqrt{\kappa H})^{-1}\ll\ell\ll1$ as $\kappa\to\infty$. Consider a  lattice $(\mathcal Q_j)_j$ of $\R^3$ generated by the cube,
$$Q_\ell=(-\ell/2,\ell/2)\times(-\ell/2,\ell/2)\times(-\ell/2,\ell/2)\,.$$
Let $\mathcal J=\{j~:~\mathcal Q_j\subset D\}$ and $N={\rm Card}\,\mathcal J$. Then, as $\kappa\to\infty$, the natural number $N$ satisfies,
\begin{equation}\label{eq-N}
N=\frac{|D|}{\ell^3}+o\left(\frac1{\ell^3}\right)\,.
\end{equation}
Moreover, we have the lower bound,
\begin{equation}\label{eq-lb1}
\mathcal E_0(\psi,\Ab;D)\geq \sum_{j\in\mathcal J}\mathcal E_0(\psi,\Ab;\mathcal Q_j)\,.
\end{equation}
For each $j\in\mathcal J$, we will bound from below the term $\mathcal E_0(\psi,\Ab;\mathcal Q_j)$.
Let $x_j$ be the center of the cube $\mathcal Q_j$. Using the estimate of $\|\Ab-\Fb\|_{C^{0,1/2}(\Omega)}$ given in Proposition~\ref{prop-ap}, we may write for all $x\in\mathcal Q_j$,
$$|\Ab(x)-\Fb(x)-(\Ab(x_j)-\Fb(x_j))|\leq C\lambda \ell^{1/2}\,,$$
where $C$ is a constant that is independent of $j$, $x$ and $\kappa$,  and the parameter $\lambda$ is defined by,
\begin{equation}\label{eq-lambda}
\lambda=\frac1{\sqrt{\kappa H}}\sqrt{\ln\sqrt{\frac{\kappa}H}}\,.
\end{equation}
We define $\varphi_j(x)=(\Ab(x_j)-\Fb(x_j))\cdot x$, $u_j(x)=e^{i\varphi(x)}\psi(x)$ and $\ab_j(x)=\Ab(x)-\nabla\varphi_j(x)$. Then we may write,
\begin{equation}\label{eq-est-mp}
\forall~x\in\mathcal Q_j\,,\quad |\ab_j(x)-\Fb(x)|\leq C\lambda \ell^{1/2}\,,
\end{equation}
and
\begin{equation}\label{eq-gauge}
\mathcal E_0(\psi,\Ab;\mathcal Q_j)=\mathcal E_0(u_j,\ab_j;\mathcal Q_j)\,.
\end{equation}
We may write, for all $\delta\in(0,1)$,
$$|(\nabla-i\kappa H\ab_j)u_j|^2\geq (1-\delta)|(\nabla-i\kappa H\Fb)u_j|^2-2\delta^{-1}(\kappa H)^2|\ab_j-\Fb_j|^2|u_j|^2\,.$$
We insert this estimate into the expression of $\mathcal E_0(u_j,\ab_j;\mathcal Q_j)$ then we use the estimate in \eqref{eq-est-mp} and that $|u_j|=|\psi|$ to get,
\begin{equation}\label{eq-lb2}
\mathcal E_0(u_j,\ab_j;\mathcal Q_j)\geq (1-\delta)\mathcal E_0(u_j,\Fb;\mathcal Q_j)-C\delta^{-1}(\kappa H)^2\lambda^2\ell\int_{\mathcal Q_j}|\psi|^2\,dx\,.
\end{equation}
Let $R=\ell\sqrt{\kappa H}$ and $b=H/\kappa$. For all $x\in\R^3$ such that $|x|\leq R$, we define,
$$v_j(x)=u\left(x_j+\frac{x}{\sqrt{\kappa H}}\right)\,.$$
Then a simple change of variable shows that,
\begin{equation}\label{eq-lb3}
\mathcal E_0(u_j,\Fb;\mathcal Q_j)=\frac1{b\sqrt{\kappa H}}\,G_{Q_R}(v_j)\,,
\end{equation}
where $G_{Q_R}$ is the functional in \eqref{eq-3D-en} and $Q_R$ is the cube in \eqref{eq-QR}.

We select  $\ell$ in the following way,
\begin{equation}\label{eq-ell}
\ell=\left(\frac{\kappa H}{\ln\kappa}\right)^{1/4}\frac1{\sqrt{\kappa H}}\,.
\end{equation}
With this choice, we have $(\sqrt{\kappa H})^{-1}\ll\ell\ll1$, $1\ll R$ and $\displaystyle\frac1{R^2}\ln\frac{R}{\sqrt{b}}\ll1$ as $\kappa\to\infty$.
Consequently, Theorem~\ref{thm-3D-gse} tells us that  the ground state  $M_0(b,R)$ in \eqref{eq-3D-gse} satisfies
$$M_0(b,R)=bR^3\ln\frac1{\sqrt{b}}\,\big(1+o(1)\big)\,.$$
 Since $v_j\in H^1(Q_R)$, we get $G_{Q_R}(v_j)\geq M_0(b,R)$. Substituting this into \eqref{eq-lb3} and using the aforementioned asymptotic expansion of $M_0(b,R)$, we get,
\begin{equation}\label{eq-lb4}
\mathcal E_0(u_j,\Fb;\mathcal Q_j)=\frac{R^3}{\sqrt{\kappa H}}\,\ln\frac1{\sqrt{b}}\,\big(1+o(1)\big)\,.
\end{equation}
By inserting \eqref{eq-lb4} into \eqref{eq-lb2} and using \eqref{eq-gauge}, we get for all $j\in\mathcal J$,
$$\mathcal E(\psi,\Ab;\mathcal Q_j)\geq
(1-\delta)\frac{R^3}{\sqrt{\kappa H}}\,\ln\frac1{\sqrt{b}}\,\big(1+o(1)\big)-C\delta^{-1}\ell(\kappa H)^2\lambda^2\int_{\mathcal Q_j}|\psi|^2\,dx\,.$$
Taking the sum over $j\in\mathcal J$ and using \eqref{eq-lb1}, we get,
\begin{equation}\label{eq-lb5}
\mathcal E(\psi,\Ab;D)\geq
(1-\delta)\, N\times\frac{R^3}{\sqrt{\kappa H}}\,\ln\frac1{\sqrt{b}}\,\big(1+o(1)\big)-C\delta^{-1}\ell(\kappa H)^2\lambda^2\int_{D}|\psi|^2\,dx\,,\end{equation}
where $N={\rm Card}\,\mathcal J$.
To finish the proof, we use the bound $|\psi|\leq 1$, the definition of $\lambda$ in \eqref{eq-lambda}, and we choose $\delta=\ell^{1/2}$. This gives that the remainder term in \eqref{eq-lb5} is equal to $o(\kappa H\ln\sqrt{\kappa/H}\,)$\,. For the leading order term in \eqref{eq-lb5}, we  use the asymptotic expansion of $N$ in \eqref{eq-N}, that $R=\ell\sqrt{\kappa H}$, and we observe that it is equal to
$$\kappa H\ln\sqrt{\frac{\kappa}H} \,\big(1+o(1)\big)\,.$$
\end{proof}

\begin{proof}[Proof of Theorem~\ref{thm-3D-main}]
Combining the upper bound in Theorem~\ref{thm-ub} and the lower bound in Theorem~\ref{thm-lb} with $D=\Omega$, we get the estimate of the ground state energy in Theorem~\ref{thm-3D-main}.
\end{proof}

\begin{proof}[Proof of Corollary~\ref{corol-main}]
The convergence of $\curl \Ab-\beta$ in $L^2(\R^3;\R^3)$ is proved in Theorem~\ref{thm-ub}.
Since $\Div(\Ab-\Fb)=0$ and $\Ab-\Fb\in \dot H^1(\R^3;\R^3)$, we get that,
$$\|\nabla\curl(\ab-\Fb)\|_{L^2(\R^3)}=\|\curl(\Ab-\Fb)\|_{L^2(\R^3)}\,.$$
Consequently, it results from the convergence of $\curl\Ab$ in $L^2(\R^3)$ that $\curl\Ab\to\beta$ in $H^1(\R^3;\R^3)$.

We prove the convergence of
$\mu_{\kappa,H}(\psi,\Ab)$.
Let $\Bb(x)=\curl\Ab(x)$. Since $\Div\Ab=0$, it results by taking the $\curl$ on both sides of the second equation in \eqref{eq-3D-GLeq},
$$-\Delta \Bb+\Bb=\mu_{\kappa,H}(\psi,\Ab)\quad{\rm in~}\,\Omega\,.
$$
Since $\Bb\to\beta$ in $H^1(\R^3;\R^3)$, we get that $-\Delta \Bb+\Bb\to\beta\,dx$ in $\mathcal D'(\R^3;\R^3)$.

It remains to prove the convergence of the measure $e_{\kappa,H}(\psi,\Ab)$. It suffices to prove
that $e_{\kappa,H}(\psi,\Ab)\to dx$ in the sense of measures. If $D$ is any open set in $\Omega$
with smooth boundary, then we have by Theorem~\ref{thm-lb}\,,
$$\mathcal E_0(\psi,\Ab;D)\geq |D|\kappa H\ln\sqrt{\frac{\kappa}{H}}\,\big(1+o(1)\big)\,,\quad
\mathcal E_0(\psi,\Ab;\Omega\setminus D)\geq |\Omega\setminus
D|\kappa H\ln\sqrt{\frac{\kappa}{H}}\,\big(1+o(1)\big)\,.$$ Here
$\mathcal E_0(\psi,\Ab;D)$ is introduced in \eqref{eq-en-E0}.
Recall the functional $\mathcal E^{\rm 3D}$ in \eqref{eq-3D-GLf}.
Since \begin{align*} \mathcal E_0(\psi,\Ab;D)+\mathcal
E_0(\psi,\Ab;\Omega\setminus D)&=\mathcal E^{\rm 3D}(\psi,\Ab)\\
&\leq \kappa
H|\Omega|\ln\sqrt{\frac{\kappa}H}\,\big(1+o(1)\big)\,,\end{align*}
 we
infer from Theorem~\ref{thm-ub},
$$ \mathcal E_0(\psi,\Ab;D)=|D|\kappa H\ln\sqrt{\frac{\kappa}{H}}\,\big(1+o(1)\big)\,.$$
This is sufficient to conclude the convergence of $e_{\kappa,H}(\psi,\Ab)$ to $dx$ in the sense of measures.
\end{proof}

\section*{Acknowledgements}
The author is partially supported by the Lundbeck foundation and by a grant from  the Lebanese University.

\end{document}